\documentclass[12pt]{amsart}
\usepackage[all]{xy}
\usepackage{graphicx}
\usepackage{tikz}  

\usepackage{amsmath}
\usepackage{amssymb}
\usepackage{amsfonts}
\usepackage{mathrsfs}
\usepackage{mathtools}

\newcommand{\Cdb}{\ensuremath{\mathbb{C}}}
\newcommand{\Ddb}{\ensuremath{\mathbb{D}}}
\newcommand{\Edb}{\ensuremath{\mathbb{E}}}

\newcommand{\Rdb}{\ensuremath{\mathbb{R}}}

\newcommand{\Zdb}{\ensuremath{\mathbb{Z}}}

\newcommand{\B}{\mbox{${\mathcal B}$}}

\newcommand{\M}{\mbox{${\mathcal M}$}}
\newcommand{\N}{\mbox{${\mathcal N}$}}

\renewcommand{\P}{\mbox{${\mathcal P}$}}

\newcommand{\R}{\mbox{${\mathcal R}$}}

\newcommand{\norm}[1]{\Vert#1\Vert}
\newcommand{\bignorm}[1]{\bigl\Vert#1\bigr\Vert}
\newcommand{\Bignorm}[1]{\Bigl\Vert#1\Bigr\Vert}

\newtheorem{theorem}{Theorem}[section]
\newtheorem{lemma}[theorem]{Lemma}

\newtheorem{proposition}[theorem]{Proposition}
\newtheorem{definition}[theorem]{Definition}
\theoremstyle{remark}
\newtheorem{remark}[theorem]{\bf Remark}
\theoremstyle{definition}

\numberwithin{equation}{section}

\begin{document}

\title[Dilation of Fourier multipliers]
{Absolute dilations of ucp self-adjoint Fourier multipliers: the non unimodular case}

\author[C. Duquet]{Charles Duquet}
\email{charles.duquet@univ-fcomte.fr}
\address{Laboratoire de Math\'ematiques de Besan\c con, 
Université Marie et Louis Pasteur, 16 route de Gray
25030 Besan\c{c}on Cedex, FRANCE}

\author[C. Le Merdy]{Christian Le Merdy}
\email{clemerdy@univ-fcomte.fr}
\address{Laboratoire de Math\'ematiques de Besan\c con, 
Université Marie et Louis Pasteur, 16 route de Gray
25030 Besan\c{c}on Cedex, FRANCE}

\date{\today}

\begin{abstract}
Let $\varphi$ be a normal semi-finite faithful weight  on a von Neumann algebra $A$,
let $(\sigma^\varphi_r)_{r\in{\mathbb R}}$ denote the modular automorphism group of $\varphi$, 
and let $T\colon A\to A$ be a linear map. 
We say that $T$ admits
an absolute dilation if there exists another von Neumann algebra $M$ equipped with a
normal semi-finite faithful weight $\psi$,
a $w^*$-continuous, unital and weight-preserving
$*$-homomorphism  $J\colon A\to M$
such that $\sigma^\psi\circ J=J\circ \sigma^\varphi$, 
as well as a weight-preserving
$*$-automorphism  $U\colon M\to M$ such that
$T^k=\Edb_JU^kJ$ for all integer $k\geq 0$, 
where $\Edb_J\colon M\to A$ is the conditional expectation associated with $J$.
Given any locally compact group $G$ and any real valued function $u\in C_b(G)$, we prove
that if $u$ induces a unital completely positive Fourier multiplier
$M_u\colon VN(G) \to VN(G)$, then $M_u$  admits
an absolute dilation. Here $VN(G)$ is equipped with its Plancherel weight $\varphi_G$. 
This result had been settled by the first named author in the case when
$G$ is unimodular so the salient point in this paper is that 
$G$ may be non unimodular, and hence $\varphi_G$ may not be a trace. The absolute dilation
of $M_u$ implies that for any $1<p<\infty$, the $L^p$-realization
of $M_u$ can be dilated into an isometry acting on a non-commutative $L^p$-space. 
We further prove that if $u$ is valued in $[0,1]$, then   the $L^p$-realization
of $M_u$ is a Ritt operator with a bounded $H^\infty$-functional calculus.
\end{abstract}

\maketitle



\vskip 1cm

\section{Introduction}\label{1Intro}    
Let $(\Omega,\mu)$ be any measure space, let  $1<p<\infty$ and let 
$T\colon L^p(\Omega)\to L^p(\Omega)$ be a positive contraction.
Akcoglu’s dilation theorem \cite{A,AS} asserts that $T$ admits an isometric dilation, that  is,
there exists a second
measure space $(\Omega',\mu')$, an isometry $J\colon L^p(\Omega)\to L^p(\Omega')$, an isometry
$U\colon L^p(\Omega')\to L^p(\Omega')$ and a contraction  $Q\colon L^p(\Omega')\to
L^p(\Omega)$ such that $T^k=QU^kJ$ for all integers $k\geq 0$. This $L^p$-analogue of Nagy's dilation theorem
has been instrumental in  various topics concerning $L^p$-operators, such as functional calculus, harmonic analysis and
ergodic theory. We refer to \cite{AFL, HVVW,KW, LMX} and the references therein for specific results.

Turning to non-commutative analysis, we note that Akcoglu’s theorem does not extend to 
non-commutative $L^p$-spaces \cite{JLM}.
This obstacle led to new, interesting difficulties 
in the development of operator theory and harmonic analysis on non-commutative $L^p$-spaces. 
Also, it became important to exhibit classes of 
completely positive contractions
on non-commutative $L^p$-spaces which do admit a dilation into an isometry acting on a
second non-commutative $L^p$-space. We refer to \cite{Ar1, Ar3, Ar2, D, DL, HRW, JLX} 
for various  results of this type.
In all these papers (except \cite{HRW}), the construction of an isometric dilation 
of a map $T\colon L^p(A)\to L^p(A)$ relies on the
existence of a trace (or weight) preserving dilation of $T$ on the underlying
von Neumann algebra $A$. This leads us to introduce the following 
definition, which extends \cite[Definition 2.1]{DL}. Let $\varphi$ be a normal semi-finite faithful weight on
$A$ (we say that $(A,\varphi)$ is a weighted von Neumann algebra in this case).
Let $(\sigma^\varphi_r)_{r\in{\mathbb R}}$ denote the modular automorphism group of $\varphi$. We say that 
a linear map $T\colon A\to A$ admits
an absolute dilation if there exists another
weighted von Neumann algebra $(M,\psi)$, a
$w^*$-continuous, unital and weight-preserving
$*$-homomorphism  $J\colon A\to M$
such that $\sigma^\psi\circ J=J\circ \sigma^\varphi$, 
as well as a weight-preserving
$*$-automorphism  $U\colon M\to M$ such that
$T^k=\Edb_JU^kJ$ for all integers $k\geq 0$, 
where $\Edb_J\colon M\to A$ is the conditional expectation associated with $J$.
This dilation property implies that $T$ is unital, completely postive and weight-preserving,
that $T$ has a natural extension to $L^p(A)$ and that the
latter admits a dilation into an isometry acting on another 
non-commutative $L^p$-space (see Proposition \ref{2LpDil}).

In \cite{DL}, we gave a characterization of bounded Schur multipliers which admit an absolute dilation.
In \cite{D}, the first named author showed that for any unimodular locally compact group $G$,
every unital completely positive Fourier multiplier $VN(G)\to VN(G)$
(in the sense of \cite{DCH}) whose symbol is real-valued admits an absolute dilation. 
The case  when $G$ is discrete had been settled in \cite{Ar1}. (The discrete case also 
follows from \cite[Corollary]{Ric} and \cite[Theorem 4.4]{HM}). The objective of this paper
is to extend the above mentioned result to the non unimodular case. 
Let $G$ be any locally compact group, let $u\colon G\to\Rdb$ be a continuous
function and assume that $u$ induces
a unital completely positive Fourier multiplier $M_u\colon VN(G)\to VN(G)$.
We show that $M_u$ admits an absolute dilation, see Theorem \ref{4AD}.

In the previous paragraph, the group von Neumann algebra $A=VN(G)$ is
equipped with its Plancherel weight $\varphi=\varphi_G$ (see
Section 3 for background). This specific weight is a trace if and only if $G$ is unimodular.
This is the reason why passing from the unimodular case (considered in \cite{D})
to the non unimodular case (considered in the present paper) requires new arguments. 
When $\varphi_G$ is a non tracial weight, we need to consider the Haagerup non-commutative
$L^p$-spaces  associated with $(VN(G),\varphi_G)$ (see Subsection \ref{22} for background). In particular, the 
space $L^1(VN(G),\varphi_G)$ plays a key role
in the proof of Theorem \ref{4AD}.

We note that a dilation theorem for semigroups of unital completely positive Fourier multiplier
with real symbols was established in \cite{Ar2}. However it cannot be applied to our case. Indeed
there is no known way to pass from a dilation result for semigroups to a dilation result for single operators.
We also note that if we restrict to the case when $G$ is unimodular,  the proof of Theorem \ref{4AD}
is partly different and somewhat 
simpler than the one
in \cite{D} or \cite{Ar1}.

In the last part of the paper (Section 5), we give an application of our dilation result to functional calculus.
We consider a unital completely positive Fourier multiplier $M_u\colon VN(G)\to VN(G)$ and we assume 
that $u$ is valued in $[0,1]$. Then we show that for any $1<p<\infty$, the 
$L^p$-realization of $T$ on $L^p(VN(G),\varphi_G)$ is a Ritt operator which admits a bounded
$H^\infty$-functional calculus in the sense of 
\cite{ALM, LM}.

\bigskip
\section{Absolute dilations and isometric $p$-dilations
in the weighted case}\label{2AD}  
We refer the reader to \cite{Tak1}
for general information on von Neumann algebras.
Given any von Neumann algebra $A$, we let 
$A_*$ denote its predual. We let 
$A^+$ and $A_{*}^+$ denote the positive cones of $A$ and $A_*$,
respectively. A map $V\colon A\to A'$ between two von Neumann algebras
is called positive if $V(A^+)\subset A'^+$. We will use the stronger notion
of completely positive map, for which we refer to \cite[Section IV.3]{Tak1}.

\smallskip
\subsection{Absolute dilations}\label{21}
We refer to  \cite{S,Tak2} for general information on weights. 
Let $A$ be a von Neumann algebra equipped
with a normal semi-finite faithful weight $\varphi$. 
In this situation, we say that $(A,\varphi)$
is a weighted von Neumann algebra. We let 
$(\sigma^\varphi_r)_{r\in{\mathbb R}}$ denote the
modular automorphism group of $\varphi$. Then we set
$$
\N_\varphi=\bigl\{x\in A\, :\, \varphi(x^*x)<\infty\bigr\}
$$
and 
$$
\M_\varphi={\rm Span}\bigl\{y^*x\, :\, x,y\in\N_\varphi\bigr\}.
$$
We recall that $\N_\varphi$ is a $w^*$-dense left ideal of $A$ and that 
$\M_\varphi$ is a $w^*$-dense 
$*$-subalgebra of $A$. Further 
$\M_\varphi={\rm Span}\{x^*x\, :\, x\in\N_\varphi\}$ and the restriction of 
$\varphi$ to the set 
$\{x^*x\, :\, x\in \N_\varphi\}$ uniquely extends to a functional (still denoted by)
$$
\varphi\colon\M_\varphi\longrightarrow \Cdb.
$$
Let $(M,\psi)$ be a second  weighted von Neumann algebra and let 
$T\colon A\to M$ be any positive map. We say that
$T$ is weight-preserving provided that 
$$
\psi\circ T=\varphi\quad\hbox{on}\ A^+.
$$

Let  $J\colon A\to M$ be a $w^*$-continuous, unital and weight-preserving
$*$-homomorphism such that 
$$
\sigma^\psi\circ J=J\circ \sigma^\varphi.
$$
Here and thereafter, this identity means that 
$\sigma^\psi_r\circ J=J\circ \sigma^\varphi_r$ on
$A$ for all $r\in\Rdb$. According to either  
\cite[Theorem 10.1]{S} or \cite[Theorem IX.4.2]{Tak2}, there
exists a unique weight-preserving conditional expectation
\begin{equation}\label{2EJ}
\Edb_J\colon M\longrightarrow  A.
\end{equation} 
That is, $\Edb_J\colon M\to A$ is the unique linear map 
satisfying $\Edb_J\circ J=I_A$, $\norm{\Edb_J}=1$ and $\varphi\circ \Edb_J=\Edb_J\circ\psi$.
We note that we further have $\sigma^\varphi\circ \Edb_J=\Edb_J\circ \sigma^\psi$, see 
\cite[Corollary 10.5]{S}, and that $\Edb_J$ is $w^*$-continuous (see also Lemma \ref{EJ=J1*}).

\begin{definition}\label{1DefAD} 
Let $(A,\varphi)$ be a weighted von Neumann algebra
and let $T\colon A\to A$ be a bounded map. 
We say that $T$ admits
an absolute dilation if there exists another
weighted von Neumann algebra $(M,\psi)$, a
$w^*$-continuous, unital and weight-preserving
$*$-homomorphism  $J\colon A\to M$
such that $\sigma^\psi\circ J=J\circ \sigma^\varphi$, 
as well as a weight-preserving
$*$-automorphism  $U\colon M\to M$ such that
\begin{equation}\label{1Formula}
\forall\, k\geq 0,\qquad T^k = \Edb_JU^k J.
\end{equation}
\end{definition}

\begin{remark} \ 

\smallskip
{\bf (1)} Assume that $T\colon (A,\varphi)\to (A,\varphi)$ admits
an absolute dilation and consider $(M,\psi)$, 
$J$ and $U$ as in Definition 
\ref{1DefAD}. In particular, we have $T=\Edb_JU J$. 
We observe that $J,U$ and $\Edb_J$ are 
both unital, completely positive and weight-preserving. 
Further $\sigma^\psi\circ U
=U\circ \sigma^\psi$, by 
\cite[Corollary VIII.1.4]{Tak2}. Since 
$\sigma^\psi\circ J=J\circ \sigma^\varphi$ and 
$\sigma^\varphi\circ \Edb_J=\Edb_J\circ \sigma^\psi$, we deduce that 
$T$ is unital, completely positive, weight-preserving and that 
$\sigma^\varphi\circ T=T\circ\sigma^\varphi$.

\smallskip 
{\bf (2)} Assume that $\varphi$ is a state and consider $T\colon (A,\varphi)\to (A,\varphi)$. If 
$T$ admits
an absolute dilation and if $(M,\psi)$ and $J\colon A\to M$ are given by 
Definition \ref{1DefAD}, then $\psi(1)=\psi(J(1))=\varphi(1)=1$, hence $\psi$ is a state.
It therefore follows from \cite[Theorem 4.4]{HM}, that 
$T$ admits
an absolute dilation if and only if 
$T$ is factorizable in the sense 
of \cite{AD, Ric}. Absolute dilations in the state case go back to \cite{K}.
We refer to 
\cite{Ar3} for various examples.

\smallskip
{\bf (3)}  Assume that $\varphi$ is a trace
and that $T\colon (A,\varphi)\to (A,\varphi)$ admits
an absolute dilation. Then the weighted von Neumann algebra $(M,\psi)$
from Definition \ref{1DefAD} can be chosen in such a way that $\psi$ is a trace.

Indeed, assume that $T$ satisfies Definition \ref{1DefAD} and let $M_0$
be the centralizer of $\psi$, that is, 
$M_0=\{m\in M\, :\, \sigma^\psi_r(x)=x\ \hbox{for all}\ r\in\Rdb\}$. Let 
$$
\B={\rm Span}\Bigl\{\sum_{k=1}^m U^{k_1}J(x_1)U^{k_2}J(x_2)\cdots U^{k_m}J(x_m)\, :\,
m\geq 1,\, k_1,\ldots,k_m\in\Zdb,\, x_1,\ldots x_m\in A\Bigr\},
$$
and let $B=\overline{\B}^{w^*}$. Since $\sigma^\varphi$ is trivial, $U^kJ(x)\in M_0$ for all 
$x\in A$ and  $k\in\Zdb$. Hence
$B$ is a sub-von Neumann algebra of $M_0$. Therefore,
the restriction $\psi_B$ of $\psi$ to $B^+$ is a normal faithful trace. Let us show its semi-finiteness.
For any $x\in \N_\varphi$ and any $k\in\Zdb$, we have 
$$
\psi\bigl((U^kJ(x))^*(U^kJ(x))\bigr)= \psi\bigl(U^kJ(x^*x)\bigr)=\varphi(x^*x)\,<\infty,
$$
that is, $U^kJ(x)\in\N_\psi$. Since $\N_\varphi$ is $w^*$-dense in $A$
and $J,U$ are $w^*$-continuous, $U^kJ(\N_\varphi)$ is $w^*$-dense in
$U^kJ(A)$. Since $\N_\psi$ is a left ideal, we deduce that
$\N_\psi\cap \B$ is $w^*$-dense in $\B$. Consequently,
$\N_\psi\cap B$ is $w^*$-dense in $B$. Thus, $\psi_B$ is semi-finite.
Consequently, $(M,\psi)$ can be replaced by $(B,\psi_B)$ in the dilation of $T$.

In the paper \cite{DL}, we considered absolute dilations in the tracial case. The above paragraph
shows that in the case when  $\varphi$ is a trace, $T\colon (A,\varphi)\to (A,\varphi)$ admits
an absolute dilation in the sense of Definition \ref{1DefAD} if and only if 
it admits
an absolute dilation in the sense of \cite{DL}.

\smallskip
{\bf (4)} Combining \cite[Section 3]{HM} with (3) above, we obtain examples
of unital, completely positive, weight-preserving maps 
$T\colon (A,\varphi)\to(A,\varphi)$ such that 
$\sigma^\varphi\circ T=T\circ\sigma^\varphi$, without any absolute dilation (see \cite[Remark 7.3]{DL} 
for details).
\end{remark}

We record for further use the so-called 
Pedersen-Takesaki theorem (see \cite[Theorem 6.2]{S}).

\begin{lemma}\label{1PT}
Let $(A,\varphi)$ be a weighted von Neumann algebra and let
$B\subset \N_\varphi$ be a $w^*$-dense $*$-algebra such that $\sigma^\varphi_r(B)\subset B$
for all $r\in\Rdb$. Let $\gamma$ be a normal semi-finite weight on $A$ such that 
$\gamma\circ\sigma^\varphi = \gamma$
and $\gamma(x^*x)=\varphi(x^*x)$  for all $x\in B$. Then $\gamma=\varphi$.
\end{lemma}

\smallskip
\subsection{Haagerup non-commutative $L^p$-spaces}\label{22}
Let $(A,\varphi)$ be a weighted von Neumann algebra. We
recall the definition of the Haagerup non-commutative $L^p$-spaces
$L^p(A,\varphi)$, as well as some of their properties. We refer the reader 
to \cite{H}, \cite[Section 1]{HJX}, \cite[Chapter 9]{Hiai}, \cite{PX} and \cite{Terp}
for details and complements. 
Assume that $A$ acts on some Hilbert space $H$ and let
$$
\mathcal R = A \rtimes_{\sigma^\varphi} \Rdb\subset 
B(L^2(\Rdb))\overline{\otimes} A\subset B\bigl(L^2(\Rdb;H)\bigr)
$$
be the crossed product associated with $\sigma^\varphi\colon\Rdb\to
{\rm Aut}(A)$, see e.g. \cite[Chapter IV]{S} or \cite[Chapter X]{Tak2}. 
We regard $A$ as a sub-von Neumann
algebra of $\R$ in the natural  way. Let 
$\widehat{\sigma}^\varphi\colon\Rdb\simeq \widehat{\Rdb}
\to {\rm Aut}(\R)$ be the dual action of $\sigma^\varphi$ 
and note that 
$\widehat{\sigma}^\varphi_r(x)=x$ 
for all $r\in\Rdb$ if and only if 
$x\in A$.

There exists a unique normal semi-finite faithful trace
$\tau_0$ on $\R$ such that 
$$
\tau_0\circ \widehat{\sigma}^\varphi_r =e^{-r}\tau_0,\qquad 
r\in\Rdb.
$$
This trace gives rise to the $*$-algebra $L^{0}(\R,\tau_0)$ of
$\tau_0$-measurable operators \cite[Chapter 4]{Hiai}. 
Then for any $1\leq p\leq\infty$, 
the Haagerup $L^p$-space $L^p(A,\varphi)$ is defined as
$$
L^p(A,\varphi) = \bigl\{
y\in L^0(\R,\tau_0)\, :\, \widehat{\sigma}^\varphi_r(y) 
=e^{-\frac{r}{p}}y\ \hbox{for all}\ r\in\Rdb\bigr\}.
$$
In particular, $L^\infty(A,\varphi)=A\subset\R$.
For all $1\leq p<\infty$,
$L^p(A,\varphi)^+= L^p(A,\varphi)\cap L^0(\R,\tau_0)^+$
is a positive cone on $L^p(A,\varphi)$.

For any normal semi-finite weight $\gamma$ on $A$, let 
$\widehat{\gamma}$ denote the dual weight on $\R$, see e.g.
\cite[Section 19]{S} or \cite[Section 10.1]{Tak2}. Let $h_\gamma$
denote the Radon-Nikodym derivative of $\gamma$ with respect to
$\tau_0$. This is the unique positive operator on $L^2(\Rdb;H)$
affiliated with  $\R$, such that $\widehat{\gamma}(y)=
\tau_0(h_\gamma^\frac12 yh_\gamma^\frac12),$ for all
$y\in\R^+$. It turns out that $h_\gamma$ belongs to
$L^1(A,\varphi)^+$ if and only if 
$\gamma\in A_{*}^+$. Further the mapping
$\gamma\mapsto h_\gamma$ on $A_{*}^+$ extends to a linear isomorphism
from $A_*$ onto $L^1(A,\varphi)$, still denoted by 
$\gamma\mapsto h_\gamma$. With this isomorphism in hands, we
equip $L^1(A,\varphi)$  with the norm 
$\norm{\,\cdotp}_1$ inherited from $A_*$, that is,
$\norm{h_\gamma}_{1}=\norm{\gamma}_{A_*}$ for all $\gamma\in A_*$.
Next, for any $1\leq p<\infty$ and any $y\in L^p(A,\varphi)$, the positive 
operator $\vert y\vert$ belongs to $L^p(A,\varphi)$ as well 
and hence $\vert y\vert^p$ belongs to $L^1(A,\varphi)$. This allows 
to define $\norm{y}_p=\norm{\vert y\vert^p}^{\frac{1}{p}}$ for all
$y\in L^p(A,\varphi)$. Then $\norm{\,\cdotp}_p$ is a complete norm on 
$L^p(A,\varphi)$.

The Banach spaces $L^p(A,\varphi)$ 
satisfy the following version of H\"older's inequality
(see e.g. \cite[Proposition 9.17]{Hiai}).
For any $1\leq p,p'\leq \infty$ 
such that $p^{-1}+p'^{-1}=1$ and for any
$x\in L^{p}(A,\varphi)$ and 
$y\in L^{p'}(A,\varphi)$, the product $xy$ 
belongs to $L^1(A,\varphi)$ and we have
$$
\norm{xy}_1\leq
\norm{x}_{p}\norm{y}_{p'}.
$$

We let ${\rm Tr}_\varphi\colon L^1(A,\varphi)\to \Cdb$ be defined by
${\rm Tr}_\varphi(h_\gamma)=\gamma(1)$ for all $\gamma\in A_*$.
This functional  has two remarkable properties. First,  for any
$x\in L^{p}(A,\varphi)$ and 
$y\in L^{p'}(A,\varphi)$ with $p^{-1}+p'^{-1}=1$, we have
${\rm Tr}_\varphi(xy) = {\rm Tr}_\varphi(yx).$ Second,
$$
{\rm Tr}_\varphi(h_\gamma x) = \gamma(x),
$$
for all $\gamma\in A_*$ and $x\in A$.
It follows from this identity and the definition of $\norm{\,\cdotp}_1$ 
that the duality pairing 
$\langle x,y\rangle ={\rm Tr}_\varphi(xy)$ for $x\in A$ and $y\in L^1(A,\varphi)$
yields an isometric isomorphism
\begin{equation}\label{2Duality}
L^1(A,\varphi)^*\simeq A.
\end{equation}
We also note that $L^2(A,\varphi)$ is a Hilbert space for the inner product
\begin{equation}\label{2Inner}
(x\vert y)_{L^2(A)}={\rm Tr}_\varphi(y^*x),\qquad x,y\in L^2(A,\varphi).
\end{equation}

We let $D_\varphi=h_\varphi$ be the Radon-Nikodym derivative
of $\varphi$ with respect to $\tau_0$. This operator
is called the density of $\varphi$. Note that it belongs to $L^1(A,\varphi)$
if and only if $\varphi$ is finite. 
It was shown in \cite[Section 2]{GL} that for any $1\leq p<\infty$,
$D_\varphi^{\frac{1}{2p}} x D_\varphi^{\frac{1}{2p}}$ belongs to 
$L^p(A,\varphi)$ for all $x\in \M_\varphi$, and that 
$D_\varphi^{\frac{1}{2p}} \M_\varphi D_\varphi^{\frac{1}{2p}}$
is dense in $L^p(A,\varphi)$. 
In the case $p=1$, we further have
\begin{equation}\label{2Tr-Varphi}
\varphi(x)={\rm Tr}_\varphi\bigl(D_\varphi^\frac12 x
D_\varphi^\frac12\bigr),
\qquad x\in\M_\varphi,
\end{equation}
see \cite[Proposition 2.13]{GL}.
We refer to \cite[Section 1]{GL} for the precise meaning of the above products.

\smallskip
\subsection{The Haagerup-Junge-Xu extension property and $p$-dilations}\label{23}
Consider two weighted von Neumann algebras  $(A,\varphi)$
and $(A',\varphi')$, and let 
$V\colon A\to  A'$ be any weight-preserving positive  operator. 
If $x\in\N_\varphi$, then $\varphi'(V(x^*x))=\varphi(x^*x)<\infty$,
hence $V(x^*x)\in\M_{\varphi'}$. Consequently,
$V$ maps $\M_\varphi$ into $\M_{\varphi'}$. Hence for any
$1\leq p<\infty$, one may define
$$
%\begin{equation}\label{2Tp}
V_p\colon D_\varphi^{\frac{1}{2p}} \M_\varphi
D_\varphi^{\frac{1}{2p}}\longrightarrow L^p(A',\varphi'),
\qquad V_p\bigl(D_\varphi^{\frac{1}{2p}} x
D_\varphi^{\frac{1}{2p}}\bigr) =D_\psi^{\frac{1}{2p}} V(x)
D_\psi^{\frac{1}{2p}},\quad \hbox{for all}\ x\in \M_\varphi.
%\end{equation}
$$
The following extension result 
is due to Haageup-Junge-Xu.

\begin{proposition}\label{2Haag} \cite[Proposition 5.4 $\&$ Remark 5.6]{HJX}
For any weight-preserving positive map $V\colon (A,\varphi)\to  (A',\varphi')$ 
and any $1\leq p<\infty$, 
$V_p$ extends to a (necessarily unique)
bounded map
from $L^p(A,\varphi)$ into $L^p(A',\varphi')$. If further 
$\norm{V}\leq 1$, then we have
$$
\bignorm{V_p\colon L^p(A,\varphi)
\longrightarrow L^p(A',\varphi')}\leq 1,\qquad 1\leq p<\infty.
$$
\end{proposition}

\begin{remark}\label{2Iso}
Consider $V\colon (A,\varphi)\to (A',\varphi')$ as in Proposition \ref{2Haag} 
and assume that
$V$ is an isometric isomorphism. Then for all
$1\leq p<\infty$,  
$V_p\colon L^p(A,\varphi)
\to L^p(A',\varphi')$ is an isometric isomorphism as well. 
To prove this, it suffices to apply Proposition \ref{2Haag} to $V$ and to $V^{-1}$, 
and to observe that $\{V^{-1}\}_p=\{V_p\}^{-1}$.
\end{remark}

\begin{definition}\label{1DefLpDil} 
Let $(A,\varphi)$ be a weighted von Neumann algebra
and let $T\colon L^p(A,\varphi)
\to L^p(A,\varphi)$ be a contraction. 
We say that $T$ admits
an isometric $p$-dilation if there exists another
weighted von Neumann algebra $(M,\psi)$, an 
isometric isomorphism $U\colon L^p(M,\psi)\to
L^p(M,\psi)$ and two contractions 
$J\colon L^p(A,\varphi)\to
L^p(M,\psi)$ and $Q\colon L^p(M,\psi)\to
L^p(A,\varphi)$ such that
$$ 
\forall\, k\geq 0,\qquad T^k = QU^kJ.
$$ 
\end{definition}

The following is a straightforward 
application of   Proposition \ref{2Haag} and Remark \ref{2Iso}. It 
extends \cite[Proposition 2.3]{DL} which was 
concerned by the tracial case only (see also \cite[p. 2281]{Ar3} for the state case).

\begin{proposition}\label{2LpDil}
Let $T\colon (A,\varphi)\to (A,\varphi)$ be a positive
weight-preserving contraction. For 
any $1\leq p<\infty$, 
let $T_p\colon L^p(A,\varphi)\to L^p(A,\varphi)$
be the contraction provided by 
Proposition \ref{2Haag}.
If $T$ admits an
absolute dilation, then $T_p$ admits an isometric
$p$-dilation.
\end{proposition}

\smallskip
\subsection{The use of analytic elements}\label{24}
Let $(A,\varphi)$ be a weighted von Neumann algebra.
We let $A_a\subset A$ be the subset of
all $x\in A$ such that $r\mapsto \sigma^\varphi_r(x)$ extends to an entire
function $z\in\Cdb\mapsto \sigma_z^\varphi(x) \in A$. Such
elements are called analytic. They play a key role in the study of 
absolute dilations. Following \cite[p.46]{GL}, we introduce
$$
\N_\varphi^\infty = \bigl\{x\in A_a\, :\, \sigma_z^\varphi(x)\in 
\N_\varphi\cap \N_\varphi^*\ \, \hbox{for all}\  z\in\Cdb\bigr\}
$$
and
$$
\M_\varphi^\infty ={\rm Span}\bigl\{y^*x\, :\, x,y\in \N_\varphi^\infty
\bigr\}.
$$
It is easy to check that $\M^\infty_\varphi\subset \N^\infty_\varphi$. 
According to  \cite[Theorem 2.4,(a)]{GL}, $\M_\varphi^\infty$
is a $w^*$-dense 
$*$-subalgebra of $A$ and by \cite[Theorem 2.4,(b)]{GL}, we have the following.

\begin{lemma}\label{2GL0} 
The space $D_\varphi^\frac12\M^\infty_\varphi D_\varphi^\frac12
\subset L^1(A,\varphi)$
is dense in $L^1(A,\varphi)$.
\end{lemma}

The following commutation property is given by \cite[Lemma 2.5]{GL}.

\begin{lemma}\label{2GL1}
For all $x\in\M_\varphi^\infty$, we have $x D_\varphi^\frac14 = 
D_\varphi^\frac14\sigma_{\frac{i}{4}}^{\varphi}(x).$
%$$
%x D_\varphi^\frac14 = 
%D_\varphi^\frac14\sigma_{\frac{i}{4}}^{\varphi}(x).
%\qquad\hbox{and}\qquad 
%x D_\varphi^\frac12 = 
%D_\varphi^\frac12\sigma_{\frac{i}{2}}^{\varphi}(x).
%$$
\end{lemma}

Next we record a variant of (\ref{2Tr-Varphi}), based on Lemma 
\ref{2GL1}, for which we refer to
\cite[Proposition 2.13]{GL}.

\begin{lemma}\label{2GL2}
For any $x\in \N_\varphi^*$ and any $y\in\M^\infty_\varphi$,
$x\sigma_{-\frac{i}{2}}^{\varphi}(y)$ belongs to $\M_\varphi$ and we have
$$
{\rm Tr}_\varphi\bigl(xD_\varphi^\frac12 y D_\varphi^\frac12\bigr)=\varphi\bigl(x
\sigma_{-\frac{i}{2}}^{\varphi}(y)\bigr).
$$
\end{lemma}

\begin{remark}\label{2Xu}
Consider $V\colon (A,\varphi)\to (A',\varphi')$ as in Proposition \ref{2Haag}.
Using (\ref{2Duality}) both  for $A$ and for $A'$, we may consider 
$V_1^*\colon A'\to A$. Then $V_1^*$ is a weight-preserving positive map.
This is proved in \cite[Proposition 5.4]{HJX} in the case when 
the weights $\varphi$ and $\varphi'$ are finite. This actually extends
to the general case \cite{X}. The proof uses the space $\N_\varphi^\infty$.
\end{remark}

\begin{lemma}\label{EJ=J1*}
Let  $J\colon (A,\varphi)\to (M,\psi)$ be a $w^*$-continuous, unital and weight-preserving
$*$-homomorphism such that 
$\sigma^\psi\circ J=J\circ \sigma^\varphi.$ Let $\Edb_J$
be the conditional expectation associated with $J$, see (\ref{2EJ}). Then
$$
\Edb_J=J_1^*.
$$
\end{lemma}

\begin{proof}
We know that $J_1^*$ is a contraction. Furthermore, $J_1^*$ is a weight-preserving
positive map, by Remark \ref{2Xu}. Hence it remains to prove that $J_1^*\circ J=I_{A}$.

Let $x\in \M_\varphi$ and let $y\in\M_{\varphi}^\infty$. Then
$$
\bigl\langle J_1^*\circ J(x), D^\frac12_\varphi y D^\frac12_\varphi \bigr\rangle_{A,L^1(A)}
=\bigl\langle  J(x), J_1\bigl(D^\frac12_\varphi y D^\frac12_\varphi\bigr) \bigr\rangle_{M,L^1(M)}
=\bigl\langle   J(x), D^\frac12_\psi J(y) D^\frac12_\psi \bigr\rangle_{M,L^1(M)}.
$$
Since $J$ is a $*$-homomorphism, we have $J(\N_\varphi)\subset \N_\psi$. 
Using the commutation property $\sigma^\psi\circ J=J\circ \sigma^\varphi$, we deduce that
$J(\N_\varphi^\infty)\subset \N_\psi^\infty$. Using again
that  $J$ is a $*$-homomorphism, we deduce that 
$$
J\bigl(\M_\varphi^\infty\bigr)\subset \M_\psi^\infty.
$$
Likewise,
$$
J\bigl(\M_\varphi\bigr)\subset \M_\psi.
$$
Hence $J(y)\in\M_\psi^\infty$ and $J(x)\in \M_\psi\subset \N_\psi^*$.
By Lemma \ref{2GL2} (applied to $\psi$), we deduce that
$$
\bigl\langle J_1^*\circ J(x), D^\frac12_\varphi y D^\frac12_\varphi \bigr\rangle_{A,L^1(A)}
={\rm Tr}_\psi\bigl(J(x)D^\frac12_\psi J(y)D^\frac12_\psi\bigr)= 
\psi\bigl(J(x)\sigma^\psi_{-\frac{i}{2}}\bigl(J(y)\bigr)\bigr).
$$
Moreover, $J(x)\sigma^\psi_{-\frac{i}{2}}\bigl(J(y)\bigr) =
J\bigl(x\sigma^\varphi_{-\frac{i}{2}}(y)\bigr)$ and 
$\psi\circ J=\varphi$, hence
$$
\bigl\langle J_1^*\circ J(x), D^\frac12_\varphi y D^\frac12_\varphi \bigr\rangle_{A,L^1(A)}
=\varphi\bigl(x\sigma^\varphi_{-\frac{i}{2}}(y)\bigr).
$$
Thus, by Lemma \ref{2GL2} again, we obtain that
$$
\bigl\langle J_1^*\circ J(x), D^\frac12_\varphi y D^\frac12_\varphi \bigr\rangle_{A,L^1(A)}
=\langle x, D^\frac12_\varphi y D^\frac12_\varphi \rangle_{A,L^1(A)}.
$$
Since $J,J_1^*$ are $w^*$-continuous and $\M_\varphi$ is $w^*$-dense, the 
result follows from Lemma \ref{2GL0}.
\end{proof}

\bigskip
\section{Background on Fourier multipliers}\label{3ourier}   
Let $G$ be a locally compact group equipped with a fixed left 
Haar measure $dg$. We let $\epsilon$ denote the unit of $G$. 
Let $\lambda\colon G\to B(L^2(G))$ be the left regular 
representation of $G$, i.e.
$$
[\lambda(g)h](t)= h(g^{-1}t),\qquad 
h\in L^2(G),\ g,t\in G.
$$
Let $VN(G)\subset B(L^2(G))$ be the von Neumann algebra 
generated by $\{\lambda(g)\, :\, g\in G\}$. For any 
$f\in L^1(G)$, we set 
$$
\lambda(f)=\int_G f(g)\lambda(g)\, dg,
$$
where the integral is defined in the strong sense. The operators
$\lambda(f)$ belong to $VN(G)$ and
$\lambda(L^1(G))$ is actually a $w^*$-dense $*$-subalgebra of 
$VN(G)$. Further for all $f_1,f_2\in L^1(G)$,
we have $\lambda(f_1\ast f_2)=\lambda(f_1)\lambda(f_2)$.

Let $\Delta_G\colon G\to(0,\infty)$ be the
modular function of $G$ and recall that it is
characterized by the identity
$$
\int_G f(g^{-1})\, dg\,=
\int_G f(g)\Delta_G(g)^{-1}\,dg.
$$
For any $f\in L^1(G)$, let $f^*\in L^1(G)$ be defined by 
$$
f^*(g) = \Delta_G(g)^{-1}\overline{f(g^{-1})},
\qquad g\in G.
$$
Then it is easy to 
check that $\lambda(f)^*=\lambda(f^*)$.

Let $K(G)$ be the space of all continuous functions
$G\to\Cdb$ with compact support. It follows from above that
$\lambda(K(G))$ is a $*$-subalgebra of $VN(G)$.
Since $K(G)$ is dense in $L^1(G)$, the latter is $w^*$-dense.

We equip $VN(G)$ with the so-called Plancherel weight 
$$
\varphi_G\colon VN(G)^+\longrightarrow [0,\infty], 
$$
for which we refer to \cite[Section 2]{Haag} 
(see also \cite[Section VII.3]{Tak2}). 
For simplicity we let
$$
\sigma^G:=\sigma^{\varphi_G}\qquad
\hbox{and}\qquad D_G:=D_{\varphi_G}
$$
denote the modular automorphism group and the density of $\varphi_G$, respectively.

We record for later use  that 
for any $f\in L^1(G)\cap L^2(G)$,
$\lambda(f)\in\N_{\varphi_G}$ and
\begin{equation}\label{3BP}
\varphi_G\bigl(\lambda(f)^*\lambda(f)
\bigr) = \norm{f}_{L^2(G)}^2,
\qquad f\in L^1(G)\cap L^2(G).
\end{equation}
Moreover if $f\in K(G)$ is such that $\lambda(f)\geq 0$, 
then $\lambda(f)\in\M_{\varphi_G}$ and
\begin{equation}\label{3e}
\varphi_G\bigl(\lambda(f)\bigr)= f(\epsilon).
\end{equation}
We also note that by 
the last but one line of \cite[p. 125]{Haag}, we have
\begin{equation}\label{3Module0}
\sigma_r^{G}(\lambda(g))=
\Delta_G(g)^{ir}\lambda(g),
\qquad g\in G, \ r\in\Rdb.
\end{equation}
This readily implies
\begin{equation}\label{3Module}
\sigma_r^G(\lambda(f)) = \lambda\bigl(\Delta_G^{ir}f\bigr),
\qquad f\in L^1(G),\ r\in\Rdb.
\end{equation}
Consequently, $\lambda(K(G))$ is $\sigma^G$-invariant.

Let $C_b(G)$ be the space of all bounded 
continuous functions $u\colon G\to\Cdb$.
Following \cite{DCH}, we say that $u\in C_b(G)$ is 
a bounded Fourier multiplier on $VN(G)$ if
there exists a $w^*$-continuous operator 
$M_u\colon VN(G)\to VN(G)$ such that 
$$
M_u(\lambda(g))=u(g)\lambda(g),\qquad 
g\in G.
$$
In this case, the operator $M_u$ is necessarily unique.
Furthermore, we have
$$
M_u(\lambda(f))=\lambda(uf),\qquad f\in L^1(G),
$$
see \cite[Proposition 1.2]{DCH} and its proof.

\begin{lemma}\label{3Preserve}
Assume that $u$ is a bounded Fourier multiplier on $VN(G)$. 
\begin{itemize}
\item [(1)] For all $r\in\Rdb$, we have $\sigma_{r}^{G}
\circ M_u = M_u\circ \sigma_{r}^{G}$.
\item [(2)] The operator $M_u$ is unital if and only if 
$u(\epsilon)=1$. 
\item [(3)] If $M_u$ is both unital and positive, then
$\varphi_G\circ M_u = \varphi_G$.
\end{itemize}
\end{lemma}

\begin{proof}
According to (\ref{3Module0}), we have
$$
\bigl(\sigma_{r}^{G}
\circ M_u\bigr)(\lambda(g))= \Delta_G(g)^{ir}u(g)
\lambda(g)=\bigl(M_u\circ \sigma_{r}^{G}
\bigr)(\lambda(g)),
$$
for all $g\in G$ and all $r\in\Rdb$. 
Since $M_u$ is $w^*$-continuous, (1) follows at once.

Part (2) is clear, since $\lambda(\epsilon)$ is the unit of $VN(G)$.

To prove (3), assume that $M_u$ is 
both unital and positive. 
Let $f\in K(G)$ such that $\lambda(f)\geq 0$. Then
$uf\in K(G)$ and $\lambda(uf)=M_u(\lambda(f))\geq 0$.
Hence by (\ref{3e}), 
$$
\varphi_G\bigl(M_u(\lambda(f)))= (uf)(\epsilon) = f(\epsilon)
=\varphi_G\bigl(\lambda(f)).
$$
This implies that for all $f_1\in K(G)$, we have 
$(\varphi_G \circ M_u)\bigl(\lambda(f_1)^*\lambda(f_1)\bigr)= 
\varphi_G\bigl(\lambda(f_1)^*\lambda(f_1)\bigr)$. 
Using part (1), the $\sigma^G$-invariance of $\lambda(K(G))$, the $w^*$-density of $\lambda(K(G))$
 and Lemma \ref{1PT}, we deduce that
$\varphi_G\circ M_u = \varphi_G$.
\end{proof}

It follows from (\ref{3Module}) that 
$$
\lambda(K(G))\subset VN(G)_a,
\quad \hbox{with}\ \sigma_z^G(\lambda(f)) = \lambda\bigl(\Delta_G^{iz}f\bigr)
$$
for all $f\in K(G)$ and all $z\in\Cdb$. This identity further
implies that 
$\lambda(f)\in \N_{\varphi_G}^\infty$. Let $K(G)\ast K(G)$ denote the linear span
of all functions $f_1\ast f_2$, with $f_1,f_2\in K(G)$. 
We deduce from above that
\begin{equation}\label{3Inc}
\lambda\bigl(K(G)\ast K(G)\bigr)\subset \M_{\varphi_G}^\infty.
\end{equation}

We will use the Haagerup non-commutative 
$L^p$-spaces $L^p(VN(G),\varphi_G)$ 
associated with the Plancherel weight of $G$. 
If $u$ is a bounded Fourier multiplier on $VN(G)$
and $M_u$ is both unital and positive, then
$\norm{M_u}\leq 1$, and hence 
\begin{equation}\label{31}
\vert u(g)\vert\leq 1,\qquad g\in G.
\end{equation}
It further follows
from Lemma \ref{3Preserve}, (3) and Proposition
\ref{2Haag} that for any $1\leq p<\infty$,
$M_u$ gives rise to a contraction
$M_{u,p}\colon L^p(VN(G),\varphi_G) \to
L^p(VN(G),\varphi_G)$
taking $D_G^{\frac{1}{2p}}xD_G^{\frac{1}{2p}}$ to
$D_G^{\frac{1}{2p}} M_u(x)
D_G^{\frac{1}{2p}}$ for all $x\in \M_{\varphi_G}$.

\begin{lemma}\label{3SA}
Assume that $u$ is a bounded Fourier multiplier 
on $VN(G)$ such that $M_u$ is both unital and positive.
The following assertions are equivalent.
\begin{itemize}
\item [(i)] The function $u$ is real-valued.
\item [(ii)] The 
operator $M_{u,2}\colon L^2(VN(G),\varphi_G)\to
L^2(VN(G),\varphi_G)$ is self-adjoint.
\end{itemize}
In this case, $M_{u,2}$ is positive in the Hilbertian sense, that is,
$(M_{u,2}(h)\vert h)_{L^2(VN(G))}\geq 0$ 
for all $h\in L^2(VN(G),\varphi_G)$, if and only if $u$ is valued in $[0,1]$.
\end{lemma}

\begin{proof}
We first note that $\varphi_G(\lambda(f')^*\lambda(f))=
(f\vert f')_{L^2(VN(G))}$
for all $f,f'\in L^1(G)\cap L^2(G)$, 
by polarisation of the identity (\ref{3BP}).
We use this identity twice below.

Let $f,f'\in K(G)\ast K(G)$ and recall  (\ref{3Inc}). By Lemma \ref{2GL1} and the fact that $M_u$ commutes with
$\sigma^G$, see Lemma \ref{3Preserve}, (1), we have
\begin{align*}
M_{u,2}\bigl(\lambda(f)D_G^\frac12\bigr) 
& = M_{u,2}\Bigl(D_G^\frac14\,\sigma^G_{\frac{i}{4}}\bigl(\lambda(f)\bigr)
D_G^\frac14\Bigr)\\
& = D_G^\frac14 M_u\Bigl(\sigma^G_{\frac{i}{4}}\bigl(\lambda(f)\bigr)\Bigr)D_G^\frac14\\
& = D_G^\frac14 \,\sigma^G_{\frac{i}{4}}\bigl(\lambda(uf)\bigr)D_G^\frac14\\
& = \lambda(uf)\,D_G^\frac12.
\end{align*}
Hence by (\ref{2Tr-Varphi}), we have 
$$
\Bigl(M_{u,2}\bigl(\lambda(f)D_G^\frac12\bigr) 
\big\vert \lambda(f')D_G^\frac12\Bigr)_{L^2(VN(G))}
 = {\rm Tr}_{\varphi_G}\Bigl(D_G^\frac12 \lambda(f')^*\lambda(uf) D_G^\frac12\Bigr)
= \varphi_G\bigl(\lambda(f')^*\lambda(uf)\bigr).
$$
Therefore,
\begin{equation}\label{2Scalar}
\Bigl(M_{u,2}\bigl(\lambda(f)D_G^\frac12\bigr) 
\big\vert \lambda(f')D_G^\frac12\Bigr)_{L^2(VN(G))}
=\int_G \overline{f'}\, uf.
\end{equation}
Likewise,
$$
\Bigl(\lambda(f)D_G^\frac12 \big\vert M_{u,2}
\bigl(
\lambda(f')D_G^\frac12\bigr)\Bigr)_{L^2(VN(G))}
=\int_G \overline{uf'}\, f.
$$
We deduce that $u$ is real-valued if and only if
$$
\Bigl(M_{u,2}\bigl(
\lambda(f)D_G^\frac12\bigr)\big\vert
\lambda(f')D_G^\frac12\Bigr)_{L^2(VN(G))}
= \Bigl(\lambda(f)D_G^\frac12 \big\vert M_{u,2}
\bigl(
\lambda(f')D_G^\frac12\bigr)\Bigr)_{L^2(VN(G))}
$$
for all $f,f'\in K(G)\ast K(G)$.

It follows from \cite[Theorem 2.4,(b)]{GL} that $\N_{\varphi_G}^\infty D_G^\frac12$ is dense
in $L^2(VN(G),\varphi_G)$. The proof of this theorem shows as well that 
$\lambda(K(G)\ast K(G)) D_G^\frac12$ is dense in  $L^2(VN(G),\varphi_G)$. 
We therefore deduce the equivalence ``$(i)\Leftrightarrow(ii)$'' from above.

By (\ref{31}), the function $u$ is valued in $[-1,1]$. Hence
the last assertion is obtained by applying (\ref{2Scalar}) with $f'=f$.
\end{proof}

\bigskip
\section{Dilation result}\label{4Main}

The main result of this paper is the following theorem. In the case
when $G$ is discrete, this result was established in 
\cite[Theorem 4.6]{Ar1}, see also \cite{Ric}. Next, it was proved in 
\cite[Theorem (B)]{D} for any unimodular group $G$.
The approach for the non-unimodular case is similar, however 
significant difficulties arise from the fact that the Plancherel weight may 
be non tracial.

\begin{theorem}\label{4AD}
Let $G$ be a locally compact  group and let 
$u\in C_b(G)$ such that 
$M_u\colon VN(G)\to VN(G)$ is 
unital and completely positive. If
$u$ is real valued, then $M_u$ admits an absolute dilation.
\end{theorem}

\begin{proof}
We let $\ell^{1,{\mathbb R}}_G$ be the real $\ell^1$-space over $G$, with standard
basis $(e_t)_{t\in G}$. Define a bilinear form
$[\,\cdotp,\,\cdotp]_u\colon \ell^{1,{\mathbb R}}_G\times \ell^{1,{\mathbb R}}_G\to\Rdb$ by
$$
[\alpha,\beta]_u=\sum_{s,t\in G} u(t^{-1}s)\alpha_t\beta_s,
\qquad \alpha=(\alpha_t)_{t\in G},\,
\beta=(\beta_s)_{s\in G}\in \ell^{1,{\mathbb R}}_G.
$$
According to \cite[Proposition 4.2]{DCH}, 
the complete positivity of
$M_u\colon VN(G)\to VN(G)$ implies that the function $u$ is positive definite. 
Hence $[\,\cdotp,\,\cdotp]_u$ is positive semi-definite. 
Let $E_u=\{\alpha\in \ell^{1,{\mathbb R}}_G\, :\, [\alpha,\alpha]_u=0\}$ 
and let $H_u$ be the real Hilbert space
obtained by completing the space $\ell^{1,{\mathbb R}}_G/E_u$ 
equipped with the norm induced by
$[\alpha,\alpha]_u(\alpha,\alpha)^\frac12$.  In the sequel,
we let $h_t\in H_u$ denote the class of 
$e_t\in\ell^1_G(\Rdb)$, for all $t\in G$. We let $(\,\cdotp\vert\,\cdotp)_{H_u}$ denote the inner
product on $H_u$.

Let $N$ be the Clifford von Neumann algebra over $H_u$
and let $\tau$ denote its standard trace. Let $N_{sa}$ be the self-adjoint part of
$N$. We recall that there exists an $\Rdb$-linear map
$$
w\colon H_u \longrightarrow N_{sa} 
$$
such that $w(\xi)$ is a unitary for all $\xi\in H_u$ with $\norm{\xi}_{H_u}=1$,
\begin{equation}\label{4Inner}
\tau\bigl(w(h)w(h')\bigr) = (h\vert h')_{H_u},\qquad h,h'\in H_u,
\end{equation}
and $N$ is generated (as a von Neumann algebra) by the range of $w$. We refer to
\cite{BKS, BR, Ric} for the construction of the Clifford von Neumann algebra and more information.

Fix some $g\in G$. For any $(\alpha_t)_{t\in G}\in \ell^{1,{\mathbb R}}_G$,
we have
\begin{align*}
\Bignorm{\sum_t \alpha_{g^{-1}t}h_t}_{H_u}^2 & = 
\sum_{s,t\in G} u(t^{-1}s) \alpha_{g^{-1}t}\alpha_{g^{-1}s} = 
\sum_{s,t\in G} u\bigl((gt)^{-1}(gs)\bigr) 
\alpha_{t}\alpha_{s} \\ & =\sum_{s,t\in G} u(t^{-1}s)\alpha_t\alpha_{s}
=\Bignorm{\sum_t \alpha_{t}h_t}_{H_u}^2.
\end{align*}
Hence there exists a (necessarily unique) orthogonal map $\theta_g\in B(H_u)$ such that 
$$
\theta_g(h_t)=h_{gt},\qquad t\in G.
$$
Let $\widetilde{\theta_g}\colon N\to N$ be the second quantization of 
$\theta_g$. By \cite[Theorem 2.11]{BKS}, $\widetilde{\theta_g}$
is a trace-preserving $*$-automorphism and we have
\begin{equation}\label{4Action}
\widetilde{\theta_g}\bigl(w(h)\bigr)= w\bigl(\theta_g(h)\bigr),\qquad
h\in H_u.
\end{equation}

Recall that $\tau$ is normalized. Let
$$
N^{\otimes} = \mathop{\overline{\otimes}}_{\mathbb Z} N,
$$
be the infinite tensor product of $(N,\tau)$ over the index set $\Zdb$, as
considered in \cite[Definition XIV.1.6]{Tak3}. Note that the associated 
state $\tau_\otimes := \mathop{\overline{\otimes}}_{\mathbb Z}\tau_N$ on $N^\otimes$
is a normal faithful tracial state. 
In the sequel, we let $1$ denote the unit of $N$ and we let $1_\otimes$ 
denote the unit of $N^\otimes$.

For any integer $m\geq 1$  and for any finite family $(x_k)_{k=-m}^{m}$ in
$N$, set
$$
\mathop{\otimes}_{k} x_k = \cdots 1\otimes 1\otimes x_{-m}\otimes\cdots\otimes
x_0\otimes x_1\otimes\ldots \otimes x_m\otimes 1\otimes 1 \cdots,
$$
where each $x_k$ is in position $k$. For any $g\in G$, we let 
$\rho(g)\colon N^\otimes\to N^\otimes$ be the 
$*$-automorphism such that 
$$
[\rho(g)]\bigl(\mathop{\otimes}_{k} x_k \bigr)=\mathop{\otimes}_{k} 
\widetilde{\theta_g}(x_{k})
$$
for all finite families $(x_k)_{k}$ in $N$. Since $\widetilde{\theta_g}$ is 
trace-preserving, $\rho(g)$ is trace-preserving.
The mapping $G\to {\rm Aut}(N)$ taking
any $g\in G$ to $\widetilde{\theta_g}$ is a group homomorphism and 
it follows from \cite[Lemma 5.4]{D}
that it is continuous in the sense of 
\cite[Proposition X.1.2]{Tak2}. We easily deduce that 
$\rho\colon G\to {\rm Aut}(N^\otimes)$ is 
a continuous group homomorphism as well.

We are going to use crossed products, for which we refer again
to either \cite[Chapter IV]{S} or \cite[Chapter X]{Tak2}. We let 
$$
M = N^\otimes\rtimes_{\rho} G\,\subset B(L^2(G))\overline{\otimes}N^\otimes
$$
be the crossed product associated with $\rho$.
We regard $N^\otimes\subset M$ in the usual way. We recall
that $VN(G)\otimes 1_\otimes\subset M$ and that $M$ is
generated (as a von Neumann algebra)
by $N^\otimes$ and $VN(G)\otimes 1_\otimes$. We also recall (see e.g. \cite[19.1.(2)]{S}) that 
\begin{equation}\label{4NewAction}
\rho(g)x=\bigl(\lambda(g)\otimes 1_\otimes\bigr)x\bigl(\lambda(g^{-1})\otimes 1_\otimes\bigr),
\qquad x\in N^\otimes,\, g\in G.
\end{equation}
This implies that the linear span
of $\{x (y\otimes 1_\otimes)\, :\, x\in N^\otimes,\, y\in VN(G)\}$ 
is $w^*$-dense  in $M$.
We will use it in the sequel without any further reference.

Let $K(G, N^\otimes)$ be the space of all
$\sigma$-strong$^*$-continuous functions $G\to N^\otimes$ with compact support. 
For any
$F\in K(G, N^\otimes)$, we define $T_F\in M$ by
$$
T_F= \int_G F(g) (\lambda(g)\otimes 1_\otimes)\, dg,
$$
where the integral is defined in the $w^*$-sense. Following
\cite[pp. 242-243]{Tak2}, given any   
$F_1,F_2\in K(G, N^\otimes)$, we may define 
$F_1\ast F_2\in K(G, N^\otimes)$  by
$$
(F_1\ast F_2)(g) = \int_{G} F_1(t)   \rho(t)\bigl(F_2(t^{-1}g)\bigr)\, dt,\qquad
g\in G,
$$
and we have
$$
T_{F_1} \ast T_{F_2} = T_{F_1\ast F_2}.
$$
Similarly, given any $F\in K(G, N^\otimes)$, we may define 
$F^*\in K(G, N^\otimes)$  by
$$
F^*(g) =\Delta_G(g)^{-1} \rho(g)\bigl(F(g^{-1})^*\bigr),\qquad g\in G,
$$
and we have
$$
T_F^*=T_{F^*}.
$$
Using these definitions, we see that 
\begin{equation}\label{4F*F}
(F^*\ast F)(g) = \int_G \Delta_G(t)^{-1}\rho(t)\bigl(F(t^{-1})^*F(t^{-1}g)\bigr)\, dt,\qquad g\in G.
\end{equation}
It actually follows from \cite[Lemma X.1.8]{Tak2} that
$$
\B : =\bigl\{T_F\, :\, F\in K(G,N^\otimes)\bigr\}
$$
is a $w^*$-dense $*$-subalgebra of $M$.

We let $\psi$ be the dual weight of $\tau_\otimes$ on $M$.  This weight is normal,
semi-finite and faithful \cite[Lemma X.1.18]{Tak2}, so that 
$(M,\psi)$ is a weighted von Neumann algebra. Moreover
$$
\B\subset \N_\psi
$$
and for any
$F\in K(G, N^\otimes)$, we have
$$
\psi\bigl(T_F^* T_F\bigr) = \tau_\otimes\bigl((F^*\ast F)(\epsilon)\bigr),
$$
see \cite[Theorem X.1.17,(i)]{Tak2}. This is an analogue 
of (\ref{3e}). According to \cite[Theorem X.1.17,(ii)]{Tak2} and (\ref{3Module0}),
the modular automorphism group $\sigma^\psi$ satisfies:

\begin{equation}\label{4Sigma-Dual}
\forall\, x\in N^\otimes,\quad
\sigma^\psi_r(x)=x
\qquad\hbox{and}\qquad
\forall\, y\in VN(G), \quad
\sigma^\psi_r(y\otimes 1_\otimes) =\sigma^G_r(y)\otimes 1_\otimes,
\end{equation}
for all $r\in\Rdb$.

We define a $w^*$-continuous unital $*$-homomorphism
$$
J\colon VN(G)\longrightarrow  M,\qquad J(y)=y\otimes 1_\otimes
\quad\hbox{for all}\  y\in VN(G).
$$
It follows from the second part of (\ref{4Sigma-Dual}) that
\begin{equation}\label{4Comm}
\sigma^\psi\circ J=J\circ \sigma^G.
\end{equation}
Let $f\in K(G)$ and let 
$F=f\otimes 1_\otimes\in K(G,N^\otimes)$. Then 
$J(\lambda(f))=T_F$ and we both have $F^*\ast F=(f^*\ast f)\otimes 1_\otimes$ and 
$J(\lambda(f)^*\lambda(f))=T_{F^*\ast F}$. Therefore,
$$
\psi\bigl(J(\lambda(f)^*\lambda(f))\bigr) = 
\tau_\otimes\bigl((F^*\ast F)(\epsilon)\bigr) = (f^*\ast f)(\epsilon) = \varphi_G
\bigl(\lambda(f)^*\lambda(f)\bigr),
$$
 by (\ref{3e}).
We deduce, using (\ref{4Comm}) and Lemma \ref{1PT},
that $J$ is weight-preserving.

Let us record two remarkable properties of $J$. First, we have
\begin{equation}\label{4Disjoint}
\psi\bigl(xJ(y)\bigr)=\tau_\otimes(x) \varphi_G(y),\qquad  x\in N^\otimes,\ y\in\M_{\varphi_G}.
\end{equation}
To prove this, consider $x\in N^\otimes\setminus\{0\}$   and define
$$
\gamma\colon VN(G)^+\longrightarrow [0,\infty],\qquad 
\gamma(y)=\tau_\otimes(x^*x)^{-1} \psi\bigl(xJ(y)x^*\bigr).
$$
This is a normal semi-finite weight.

Given any $f\in K(G)$, 
%By the first part of (\ref{4Sigma-Dual}), $x^*$ belongs to 
%the centralizer of $\psi$ hence
%$$
%\gamma\bigl(\lambda(f)\lambda(f)^*\bigr)=
%\psi\bigl(x^*xJ\bigl(\lambda(f)\lambda(f)^*\bigr)\bigr) = \psi \bigl(xJ\bigl(\lambda(f)\lambda(f)^*\bigr) x^*\bigr),
%$$
%by \cite[Theorem VIII.2.6]{Tak2}.
let $F\in K(G,N^\otimes)$ be defined by
$$
F(g) = \Delta_{G}(g)^{-1}\overline{f(g^{-1})}\rho(g)(x^*),\qquad g\in G.
$$
This function is chosen in such a way that $F^*=f\otimes x$. Thus, we have $T_F^*=T_{F^*} = xJ(\lambda(f))$. Consequently,
$$
\gamma\bigl(\lambda(f)\lambda(f)^*\bigr)=\tau_\otimes(x^*x)^{-1} 
\psi\bigl(T_{F}^*T_F\bigr) = \tau_\otimes(x^*x)^{-1} 
\tau_\otimes\bigl((F^*\ast F)(\epsilon)\bigr).
$$
Using (\ref{4F*F}), we compute 
$$
(F^*\ast F)(\epsilon) = \norm{f}_2 \,xx^*.
$$
We deduce that 
$$
\gamma\bigl(\lambda(f)\lambda(f)^*\bigr) = \varphi_G\bigl(\lambda(f)\lambda(f)^*\bigr).
$$
By the first part of (\ref{4Sigma-Dual}), $x$ belongs to the centralizer of $\psi$, hence
$$
\gamma(\,\cdotp)=\tau_\otimes(x^*x)^{-1} \psi\bigl(J(\,\cdotp)x^*x\bigr),
$$
by \cite[Theorem VIII.2.6]{Tak2}. Applying
(\ref{4Comm}), we deduce that $\gamma\circ\sigma^G=\gamma$.
Then arguing as in the proof of Lemma \ref{3Preserve}, (3), we infer  that $\gamma=\varphi_G$.
Using the fact that $x^*$ belongs to the centralizer of $\psi$, this yields
$\psi\bigl(x^*x J(y)\bigr)=\tau_\otimes(x^*x) \varphi_G(y)$
for all $y\in\M_{\varphi_G}$. The identity (\ref{4Disjoint}) follows at once.

Let us now show that
\begin{equation}\label{4Magic}
\Edb_J\bigl(xJ(y))= \tau_\otimes(x)y,\qquad x\in N^\otimes,\ y\in VN(G).
\end{equation}
Since $J$ and $\Edb_J$ are $w^*$-continuous, we may assume in the proof of (\ref{4Magic}) 
that $y\in\M_{\varphi_G}$. Let $m\in \M_{\varphi_G}^\infty$.
Since $J$ is a weight-preserving $*$-homomorphism and $J$ satisfies (\ref{4Comm}), we easily obtain, as in the proof of
Lemma \ref{EJ=J1*}, that
$$
J(y)\in  \M_{\psi}
\qquad\hbox{and}\qquad
J(m)\in  \M_{\psi}^\infty.
$$
Since $x$ belongs to the centralizer of $\psi$, the product $xJ(y)$ belongs to  $\M_{\psi}$, by \cite[Theorem VIII.2.6]{Tak2}.
Using Lemmas \ref{EJ=J1*} and \ref{2GL2}, we deduce that
\begin{align*}
\bigl\langle \Edb_J\bigl(xJ(y)), D_G^\frac12 m D_G^\frac12\bigr\rangle_{VN(G), L^1(VN(G))}
& = \bigl\langle xJ(y), J_1\bigl(D_G^\frac12 m D_G^\frac12\bigr)\bigr\rangle_{M, L^1(M)}\\
& = \bigl\langle xJ(y), D_\psi^\frac12 J(m) D_\psi^\frac12\bigr)\bigr\rangle_{M, L^1(M)}\\
& = {\rm Tr}_\psi\bigl(xJ(y)D_\psi^\frac12 J(m) D_\psi^\frac12\bigr)\\
&=\psi\bigl(xJ(y)\sigma^\psi_{-\frac{i}{2}}(J(m))\bigr).
\end{align*}
Using (\ref{4Comm}) again and (\ref{4Disjoint}), we have
$$
\psi\bigl(xJ(y)\sigma^\psi_{-\frac{i}{2}}(J(m))\bigr)
=\psi\bigl(xJ \bigl(y\sigma^G_{-\frac{i}{2}}(m)\bigr)\bigr) = \tau_\otimes (x) \varphi_G\bigl(
y\sigma^G_{-\frac{i}{2}}(m)\bigr).
$$
By Lemmas \ref{EJ=J1*} and \ref{2GL2} again, 
$$
\varphi_G\bigl(y\sigma^\psi_{-\frac{i}{2}}(m)\bigr)= {\rm Tr}_{\varphi_G}\bigl(yD_G^\frac12 mD_G^\frac12\bigr)
=\bigl\langle \Edb_J(y), D_G^\frac12 m D_G^\frac12\bigr\rangle_{VN(G), L^1(VN(G))}.
$$
Altogether, we obtain that
$$
\bigl\langle \Edb_J\bigl(xJ(y)), D_G^\frac12 m D_G^\frac12\bigr\rangle_{VN(G), L^1(VN(G))}
= \tau_\otimes (x) \bigl\langle \Edb_J(y), D_G^\frac12 m D_G^\frac12\bigr\rangle_{VN(G), L^1(VN(G))}.
$$
Since $D_G^\frac12 \M_{\varphi_G}^\infty D_G^\frac12$ is dense in $L^1(VN(G))$, see Lemma \ref{2GL0},
we deduce (\ref{4Magic}).

We now aim at constructing a suitable $U\colon M\to M$. We proceed in two steps.

We let ${\mathfrak s}\colon N^\otimes\to N^\otimes$ be the 
$*$-automorphism such that 
$$
{\mathfrak s}\bigl(\mathop{\otimes}_{k} x_k \bigr)=\mathop{\otimes}_{k} x_{k-1}
$$
for all finite families $(x_k)_{k}$ in $N$. Clearly ${\mathfrak s}$ is trace-preserving, i.e.
\begin{equation}\label{4TP}
\tau_\otimes \circ {\mathfrak s} = \tau_\otimes.
\end{equation}
Moreover $\rho(g)\circ {\mathfrak s}={\mathfrak s}\circ \rho(g)$ for 
all $g\in G$. This commutation property ensures, by \cite[Theorem 4.4.4]{Sunder}, 
that there exists a 
(necessarily unique) $*$-automorphism 
$\widetilde{\mathfrak s}\colon M\to M$ such that
$$
\forall\, x\in N^\otimes,\quad
\widetilde{\mathfrak s}(x) = {\mathfrak s}(x)
\qquad\hbox{and}\qquad
\forall\, y\in VN(G), \quad \widetilde{\mathfrak s}(y\otimes 1_\otimes) 
=y\otimes 1_\otimes.
$$It follows from these identities and (\ref{4Sigma-Dual}) that
\begin{equation}\label{4Comm-bis}
\sigma^\psi\circ \widetilde{\mathfrak s}=\widetilde{\mathfrak s}\circ \sigma^\psi.
\end{equation}
Let us  show that $\widetilde{\mathfrak s}$ is weight-preserving. Let 
$F\in K(G,N^\otimes)$ and let $r\in\Rdb$. Then
$$
\sigma^\psi_r(T_F)= \int_G \sigma^\psi_r(F(g)) 
\sigma^\psi_r(\lambda(g)\otimes 1_\otimes)\, dg\,
= \int_G F(g) \Delta_G(g)^{ir}(\lambda(g)\otimes 1_\otimes)\, dg\,
=T_{ \Delta_G^{ir}F}.
$$
Hence $\B$ is $\sigma^\psi$-invariant. Using Lemma \ref{1PT}, (\ref{4Comm-bis}), 
and the fact that $\B$ 
is a $w^*$-dense $*$-subalgebra of $M$, 
it therefore suffices to show that
\begin{equation}\label{400}
(\psi\circ \widetilde{\mathfrak s})\bigl(T_F^*T_F\bigr) = 
\psi\bigl(T_F^*T_F\bigr),\qquad F\in K(G,N^\otimes),
\end{equation}
to obtain that $\psi\circ\widetilde{\mathfrak s}=\psi$.
To prove this, consider $F\in K(G,N^\otimes)$ and write
$$
\widetilde{\mathfrak s}\bigl(T_F^*T_F\bigr)
=\widetilde{\mathfrak s}\Bigl(\int_G (F^*\ast F)(g)(\lambda(g)\otimes 1_\otimes)\, dg\Bigr)\ 
=\int_G {\mathfrak s}\bigl((F^*\ast F)(g)\bigr)(\lambda(g)\otimes 1_\otimes)\, dg.
$$
Since ${\mathfrak s}\circ\rho(t) = \rho(t)\circ{\mathfrak s}$, we deduce, using (\ref{4F*F}), that 
\begin{align*}
{\mathfrak s}\bigl((F^*\ast F)(g)\bigr) &=
 \int_G \Delta_G(t)^{-1}\rho(t)\bigl({\mathfrak s}\bigl(F(t^{-1})^*\bigr)
{\mathfrak s}\bigl(F(t^{-1}g)\bigr)\bigr)\, dt\\
& = \bigl[({\mathfrak s}\circ F)^*\ast ({\mathfrak s}\circ F)\bigr](g),\qquad g\in G.
\end{align*}
Consequently,
\begin{equation}\label{4Compo}
\widetilde{\mathfrak s}\bigl(T_F^*T_F\bigr) =T_{{\mathfrak s}\circ F}^* T_{{\mathfrak s}\circ F}.
\end{equation}
We deduce, using  (\ref{4TP}), that 
\begin{align*}
(\psi\circ \widetilde{\mathfrak s})\bigl(T_F^*T_F\bigr) & = \tau_\otimes\bigl(
[{\mathfrak s}\circ F)^*\ast ({\mathfrak s}\circ F)](\epsilon)\bigr)\\
& = \tau_\otimes\bigl({\mathfrak s}\bigl((F^*\ast F)(\epsilon)\bigr)\bigr)\\
& = \tau_\otimes\bigl((F^*\ast F)(\epsilon)\bigr)\\
& = \psi\bigl(T_F^*T_F\bigr),
\end{align*}
which proves (\ref{400}).

For any $g\in G$, let $\delta_g\in N^\otimes$ be defined by 
$$
\delta_g= \cdots  1\otimes\cdots\otimes 1\otimes
w(h_g)\otimes 1\otimes\cdots \otimes 1\otimes \cdots,
$$
where $w(h_g)$ is in position $0$. 
By assumption,
$M_u$ is unital hence $\norm{h_\epsilon}_{H_u}= u(\epsilon)=1$, see Lemma \ref{3Preserve}, (2).
Consequently, $w(h_\epsilon)$ is a unitary of $N$ hence $\delta_\epsilon$ is a self-adjoint unitary of $N^\otimes\subset M$.
Let $U\colon M\to M$ be defined by
$$
U(z) = \delta_\epsilon\,\widetilde{\mathfrak s}(z) \delta_\epsilon,\qquad z\in M.
$$
Then $U$ is a $*$-automorphism. By the first part of (\ref{4Sigma-Dual}), $\delta_\epsilon$ belongs to the centralizer 
of $\psi$, hence we may write $\psi(U(y)) = \psi(\widetilde{\mathfrak s}(y)\delta_\epsilon^2)
=\psi(\widetilde{\mathfrak s}(y)) = \psi(y)$. Thus we have
$$
\psi\circ U = \psi.
$$

Let us now check that $T=M_u$ satisfies (\ref{1Formula}).
We observe that by (\ref{4Action}), we have 
$$
\rho(g)(\delta_t)= \delta_{gt},\qquad g,t\in G.
$$
Applying (\ref{4NewAction}), we deduce that 
\begin{equation}\label{4Permute}
(\lambda(g)\otimes 1_\otimes) \delta_\epsilon 
=  \delta_g (\lambda(g)\otimes 1_\otimes),\qquad g\in G.
\end{equation}
Let $g\in G$. Then 
$\widetilde{\mathfrak s}(\lambda(g)\otimes 1_\otimes)= \lambda(g)\otimes 1_\otimes,$
hence by (\ref{4Permute}), we have 
$$
UJ(\lambda(g)) = \delta_\epsilon\delta_g (\lambda(g)\otimes 1_\otimes).
$$
Next,
$$
\widetilde{\mathfrak s}(\delta_\epsilon\delta_g) = \cdots  1\otimes\cdots\otimes 1\otimes
w(h_\epsilon)w(h_g)\otimes 1\otimes\cdots \otimes 1\otimes \cdots,
$$
where $w(h_\epsilon)w(h_g)$ is in position $1$. Hence $\delta_\epsilon$ commutes with 
$\widetilde{\mathfrak s}(\delta_\epsilon\delta_g)$ and using (\ref{4Permute}) again, we obtain
$$
U^2J(\lambda(g)) = \delta_\epsilon \widetilde{\mathfrak s}(\delta_\epsilon\delta_g) ( \lambda(g)\otimes 1_\otimes)\delta_\epsilon
= \widetilde{\mathfrak s}(\delta_\epsilon\delta_g)\delta_\epsilon\delta_g (\lambda(g)\otimes 1_\otimes).
$$
More explicitly,
$$
U^2J(\lambda(g)) = \bigl(\cdots\otimes 1\otimes w(h_\epsilon)w(h_g)\otimes
w(h_\epsilon)w(h_g)\otimes 1\otimes \cdots\bigr)(\lambda(g)\otimes 1_\otimes),
$$
with $w(h_\epsilon)w(e_g)$ in positions $0$ and $1$.
Then we obtain by induction that for all integer $k\geq 0$, we have
$$
U^kJ(\lambda(g)) = \bigl(\cdots\otimes 1\otimes w(h_\epsilon)w(e_g)\otimes\cdots\otimes
w(h_\epsilon)w(e_g)\otimes 1\otimes\cdots\bigr)(\lambda(g)\otimes 1_\otimes),
$$
with $w(h_\epsilon)w(e_g)$ in positions $0,\ldots, k-1$.
According to (\ref{4Magic}) and (\ref{4Inner}), this implies that 
\begin{align*}
\Edb_J U^k J(\lambda(g)) & = \tau_\otimes
\bigl(\cdots\otimes 1\otimes w(h_\epsilon)w(e_g)\otimes\cdots\otimes
w(h_\epsilon)w(e_g)\otimes 1\otimes\cdots\bigr)\lambda(g)\\
& = \tau\bigl(w(h_\epsilon)w(h_g)\bigr)^k \lambda(g)\\
& = (h_\epsilon\vert h_g)_{H_u}^k \lambda(g)\\
& = u(g)^k\lambda(g).
\end{align*}
Thus,  $\Edb_J U^k J(\lambda(g))=M_u^k(\lambda(g))$ for all $k\geq 0$ and all  $g\in G$.
Since the operators $M_u, J, U$ and $\Edb_J$ are all $w^*$-continuous and
${\rm Span}\{\lambda(g)\, :\, g\in G\}$ is $w^*$-dense in $VN(G)$, we deduce (\ref{1Formula}) for all $k\geq 0.$
\end{proof}

If we restrict to the case when $G$ is unimodular, the above proof is different and somewhat 
simpler than the one
in \cite{D} or \cite{Ar1}. Indeed, we do not use \cite[Lemma 2.7]{D}.

\begin{remark}\label{4Multi}
We observe (as in \cite[Theorem 6.3]{D}) that we can simulteously dilate any finite family of
Fourier multipliers satisfying the assumptions of Theorem \ref{4AD}. More explicitly,
let $u_1,\ldots,u_n\in C_b(G)$ such that 
$M_{u_j}\colon VN(G)\to VN(G)$ is 
unital and completely positive and
$u_j$ is real valued, for all $j=1,\ldots,n$. Then
 there exists a
weighted von Neumann algebra $(M,\psi)$, a
$w^*$-continuous, unital and weight-preserving
$*$-homomorphism  $J\colon VN(G)\to M$
such that $\sigma^\psi\circ J=J\circ \sigma^\varphi$, 
as well as a commuting family $(U_1,\ldots,U_n)$
of weight-preserving
$*$-automorphisms  on $M$,  such that
$$
\forall\, k_1,\ldots,k_n\geq 0,\qquad M_{u_1}^{k_1}\cdots  M_{u_n}^{k_n}= \Edb_J\,  U_1^{k_1}\cdots  U_n^{k_n} J.
$$

\end{remark}

\bigskip
\section{An application to functional calculus}\label{5FC}   
The aim of this section is to establish a functional calculus property of the 
$L^p$-realization of certain Fourier multipliers  $M_u\colon VN(G)\to VN(G)$.

Let $\Ddb=\{z\in \Cdb\, :\, \vert z \vert<1\}$ be the open unit disc. Let
$X$ be a Banach space and let $T\in B(X)$. We say that $T$ is power bounded 
if there exists a constant $C_0\geq 1$ such that $\norm{T^k}\leq C_0$ 
for all integer $k\geq 0$. Let $\sigma(T)$ denote the spectrum of $T$.
We say that $T$ is a Ritt operator if $\sigma(T)\subset\overline{\Ddb}$ and 
there exists a constant $C>0$ such that
$$
\bignorm{(z-T)^{-1}}\leq\,\frac{C}{\vert z-1\vert}\,,\qquad z\in\Cdb\setminus\overline{\Ddb}.
$$
A classical result (see e.g. \cite{Vitse}) asserts that $T$ is a Ritt operator if and only if 
$T$ is power bounded and 
there exists a constant $C_1\geq 1$ such that 
$$
n \bignorm{T^{n}- T^{n-1}}\leq C_1,\qquad n\geq 1.
$$

For any $\gamma\in\bigl(0,\frac{\pi}{2}\bigr)$, let $B_\gamma$ be the 
interior of the convex hull of $1$ and the disc $\{\vert z\vert <\sin\gamma\}$, see Figure \ref{f1} below. 
It is well-known that the spectrum of any Ritt operator $T$ is included in the 
closure of one of the $B_\gamma$. 

\begin{figure}[ht]
\vspace*{2ex}
\begin{center}
\includegraphics[scale=0.4]{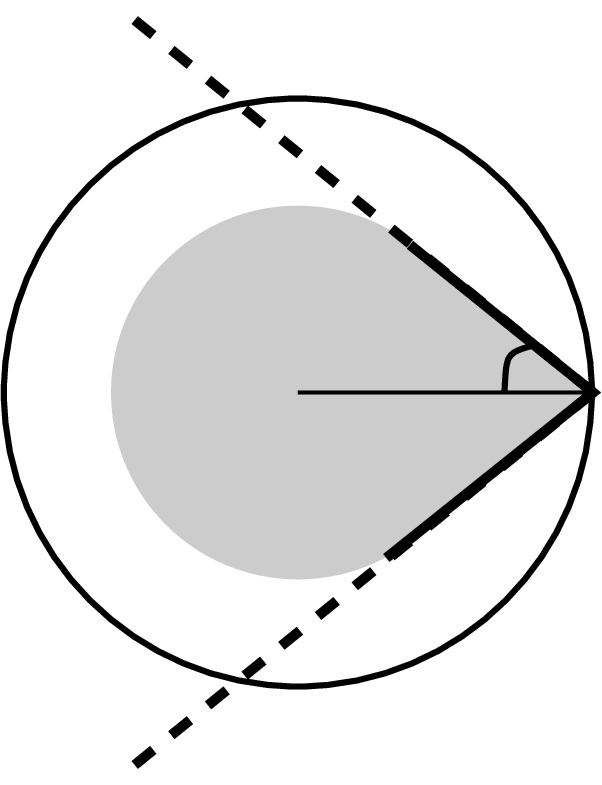}
\begin{picture}(0,0)
\put(-2,65){{\footnotesize $1$}}
\put(-68,65){{\footnotesize $0$}}
\put(-32,81){{\footnotesize $\gamma$}}
\put(-55,85){{\small $B_\gamma$}}
\end{picture}
\end{center}
\caption{\label{f1} Stolz domain}
\end{figure}
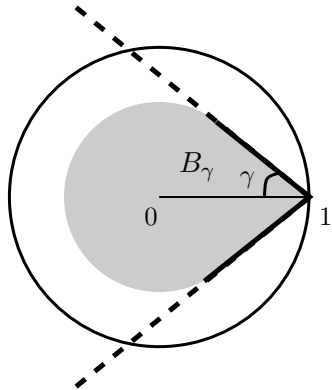

Let $\P$ be the algebra of complex polynomials and let $T\in B(X)$ be a Ritt operator.
For any $\gamma\in\bigl(0,\frac{\pi}{2}\bigr)$, we say that $T$ admits a bounded
$H^\infty(B_\gamma)$-functional calculus if there exists a constant $K\geq 1$ such that
$$
 \norm{F(T)}\leq K\sup\bigl\{\vert F(z)\vert\, :\, z\in B_\gamma\bigr\},\qquad F\in\P.
$$
Then we simply say that $T$ admits a bounded
$H^\infty$-functional calculus if it admits a bounded
$H^\infty(B_\gamma)$-functional calculus for some $\gamma\in\bigl(0,\frac{\pi}{2}\bigr)$.
This implies that $T$ is polynomially bounded, that is, $T$ satisfies an estimate
$\norm{F(T)}\leq K\sup\bigl\{\vert F(z)\vert\, :\, z\in \Ddb\bigr\}$ for all $F\in\P$.

Ritt operators and their $H^\infty$-functional calculus have attracted a lot of attention in the last 
twenty years. We refer the reader to \cite{ALM, AFL, Blu, Dun, GT, Haak, HH, LM, Vitse} 
and the reference therein for more information.

The main result of this section is the following theorem.
Note that in this statement, the multiplier $u$ is not only a real 
valued function but a nonnegative one.

\begin{theorem}\label{5FC-Ritt}
Let $G$ be a locally compact group, let 
$u\colon G\to [0,1]$ be a continuous function, and assume that 
$M_u\colon VN(G)\to VN(G)$ is 
unital and completely positive. Then for all 
$1<p<\infty$, $M_{u,p}\colon
L^p(VN(G),\varphi_G)\to L^p(VN(G),\varphi_G)$
is a Ritt operator which admits a bounded
$H^\infty$-functional calculus.
\end{theorem}

We need sectorial operators and their  $H^\infty$
functional calculus. We refer to
\cite[Chapter 10]{HVVW} for background and undefined terminology 
on this topic. As in this reference, we set 
$$
\Sigma_\theta=\bigl\{z\in\Cdb\setminus\{0\}\, :\, \vert{\rm Arg}(z)\vert<\theta\},
$$
for any $\theta\in(0,\pi)$.
The following was proved in  \cite[Proposition 4.1]{LM}.

\begin{proposition}\label{5Ritt-Sect}
Let $T\in B(X)$ be a Ritt operator. Then $I_X-T$ is a sectorial operator. If there exists 
$\theta\in\bigl(0,\frac{\pi}{2}\bigr)$ such that $I_X-T$ admits a bounded
$H^\infty(\Sigma_\theta)$ functional calculus, then $T$ admits a 
bounded $H^\infty$-functional calculus.
\end{proposition}

The proof of Theorem \ref{5FC-Ritt} will also rely on interpolation. 
Given any interpolation couple 
$(X_0,X_1)$ of Banach spaces  and $\alpha\in[0,1]$, we let
$[X_0,X_1]_\alpha$
denote the Banach space obtained by the complex interpolation
method. We refer to \cite[Chapter 4]{BL} for background.

\begin{proposition}\label{5Ritt-Inter}
Let $(X_0,X_1)$ be an interpolation couple of Banach spaces and consider an operator
$T\colon X_0+X_1\to X_0+X_1$ such that 
$$
T\colon X_0\longrightarrow X_0
\qquad\hbox{and}\qquad
T\colon X_1\longrightarrow X_1
$$
are bounded. If $T\colon X_0\to X_0$ is a Ritt operator and $T\colon X_1\to X_1$
is power bounded, then for all $\alpha\in[0,1)$, 
$T\colon [X_0,X_1]_\alpha\to [X_0,X_1]_\alpha$
is a Ritt operator.
\end{proposition}

\begin{proof}
This result was proved by Blunck in the case when $X_0=L^p(\Omega)$
and $X_1=L^q(\Omega)$ for a measure space $(\Omega,\mu)$ and  $1\leq p,q\leq \infty$, 
see \cite[Theorem 1.1]{Blu}. His arguments apply as well in this abstract setting.
\end{proof}

\begin{proof}[Proof of Theorem \ref{5FC-Ritt}]
We set $A=VN(G)$, $\varphi=\varphi_G$ and $D=D_G$ for simplcity.
We use complex interpolation for the spaces $L^p(A,\varphi)$. Following
\cite{Terp2} (see also \cite[Section 6]{CDS}, \cite{Ko} and \cite[Chapter 9]{Hiai}), 
we regard $(A, L^1(A,\varphi))$ as 
an interpolation couple of Banach spaces, using the embedding 
$\M_\varphi\to L^1(A,\varphi)$ which takes $x$ to $D^\frac12 x D^\frac12$ for any
$x\in\M_\varphi$. Then for any $1\leq p\leq \infty$, the 
mapping $D^{\frac{1}{2p}}\M_\varphi D^{\frac{1}{2p}} \to L^1(A,\varphi)$
which takes $D^{\frac{1}{2p}}xD^{\frac{1}{2p}}$ to $D^\frac12 x D^\frac12$ for any
$x\in\M_\varphi$ extends to an isometric isomorphism
\begin{equation}\label{5Terp}
L^p(A,\varphi)\simeq\bigl[A, L^1(A,\varphi)\bigr]_{\frac{1}{p}}.
\end{equation}
Furthermore, if $V\colon A\to A$ is a weight-preserving positive map, we can define
$T\colon A+L^1(A,\varphi)\to A+L^1(A,\varphi)$ 
by $T(x+y)=V(x)+V_1(y)$ for all $(x,y)\in A\times L^1(A,\varphi)$ and the interpolation
operator on $[A, L^1(A,\varphi)]_{\frac{1}{p}}$ induced by $T$ coincides with
$V_p$ though the identification (\ref{5Terp}).

Let $u\colon G\to[0,1]$ be a continuous function such that 
$M_u\colon VN(G)\to VN(G)$ is unital and completely positive and let us apply 
the above properties with $V=M_u$. We fix some $p\in [2,\infty)$, the case when 
$p\in(1,2)$ being similar.

According to Lemma \ref{3SA}, $M_{u,2}$ is positive in the Hilbertian sense. Hence
$\sigma(M_{u,2})\subset [0,1]$. Then, by the functional calculus of self-adjoint operators,
$M_{u,2}$ is a Ritt operator. Indeed, for any $z\in\Cdb\setminus\overline{\Ddb}$, we have an estimate
$$
\bignorm{(z-M_{u,2})^{-1}}\leq \sup\bigl\{\vert z-t\vert^{-1}\, :\, t\in[0,1]\bigr\}\lesssim\vert z-1\vert^{-1}.
$$
By (\ref{5Terp}) and the  reiteration theorem for complex interpolation
(see \cite[Theorem 4.6.1]{BL}), we have
$$
L^p(A,\varphi) \simeq \bigl[A, L^2(A,\varphi)\bigr]_{\frac{2}{p}}.
$$
Hence by Proposition \ref{5Ritt-Inter}, $M_{u,p}$ is a Ritt operator.

The rest of the proof is similar to the one of \cite[Theorem 7.1]{Ar3}.
Let $q\in(p,\infty)$ and let $\theta_0\in\bigl(\frac{\pi}{2},\pi\bigr)$. 
By Theorem \ref{4AD} and Proposition \ref{2LpDil}, the operator
$M_{u,q}$ admits an isometric $q$-dilation. Hence by \cite[Proposition 6.2]{ALM}, the sectorial
operator
$I_{L^q} -M_{u,q}$ admits a bounded $H^\infty(\Sigma_{\theta_0})$ functional calculus.
(The non-commutative $L^q$-spaces considered in \cite[Section 6]{ALM} are the tracial ones. 
However the results contained therein extend verbatim to the framework of Haagerup
non-commutative $L^q$-spaces, using the fact that for $1<q<\infty$, these 
spaces satisfy the UMD property, see  \cite[Corollary 7.7]{PX}.)

Let $\theta_1\in \bigl(0,\frac{\pi}{2}\bigr)$.
The operator $I_{L^2} - M_{u,2}$ is self-adjoint and $\sigma(I_{L^2} - M_{u,2})\subset [0,1]$.
Hence using again the functional calculus of self-adjoint operators, we see that
$I_{L^2} -M_{u,2}$ admits a bounded $H^\infty(\Sigma_{\theta_0})$ functional calculus.
Applying again the  reiteration theorem for complex interpolation, we obtain that 
$$
L^p(A,\varphi) \simeq \bigl[L^q(A,\varphi), L^2(A,\varphi)\bigr]_{\alpha},
\qquad\hbox{with}\ \frac{1-\alpha}{q} \, +\, \frac{\alpha}{2}\,=\,\frac{1}{p}.
$$
Now applying the interpolation theorem for the $H^\infty$-functional calculus
of sectorial operators \cite[Proposition 4.9]{KKW} (see also 
\cite[Proposition 5.8]{JLX}), we deduce that 
$I_{L^p} -M_{u,p}$ admits a bounded $H^\infty(\Sigma_{\theta})$ functional calculus,
with
$\theta = (1-\alpha)\theta_0+\alpha\theta_1.$
We can choose $q,\theta_0,\theta_1$ so that $\theta\in\bigl(0,\frac{\pi}{2})$.
(We can actually obtain  bounded $H^\infty(\Sigma_{\theta})$ functional calculus for
any $\theta>\bigl(\frac12  -\frac{1}{p}\bigr)\pi$.)
By Proposition \ref{5Ritt-Sect}, this implies that 
$M_{u,p}$ admits a bounded $H^\infty$-functional calculus.
\end{proof}

\bigskip\noindent
{\bf Acknowledgements:}
The two authors were supported by the ANR project Noncommutative analysis on groups
and quantum groups (No./ANR-19-CE40-0002). 

\smallskip\noindent
{\bf Competing interest declaration:}
The authors declare that they have no known competing financial interests or personal relationships that 
could have appeared to influence the work reported in this paper.

\vskip 1cm

\bibliographystyle{abbrv}

\end{document}